\documentclass[a4paper,11pt, draft]{article}

\usepackage{stmaryrd}
\usepackage{amsmath}
\usepackage{amsthm}
\usepackage{amstext}
\usepackage{amsfonts}
\usepackage{cite}
\usepackage{a4wide}

\usepackage{graphicx}
\usepackage[english,francais]{babel}

\newtheorem{thm}{Theorem}
\newtheorem{lemma}{Lemma}
\newtheorem{e-proposition}{Proposition}

\newtheorem{e-definition}{Definition\rm}
\newtheorem{remark}{Remark\/}
\newtheorem{exemple}{\it Example\/}

\def\Div{\mathrm{div}}
\def\s{\mathrm{sign}}

\title{Exact controllability to trajectories for entropy solutions to scalar conservation laws in several space dimensions}

\author{Carlotta Donadello$^1$ \and Vincent Perrollaz$^2$}

\begin{document}
\maketitle
\footnotetext[1]{Universit\'e de Bourgogne Franche-Comt\'e, Laboratoire de Math\'ematiques, CNRS
UMR6623, 16 route de Gray, 25000 Besan\c{c}on, \texttt{carlotta.donadello@univ-fcomte.fr}}
\footnotetext[2]{Institut Denis Poisson, Universit\'e de Tours, CNRS UMR7013, Université
d'Orléans, Parc de Grandmont, 37000 Tours \texttt{vincent.perrollaz@lmpt.univ-tours.fr}}

\begin{abstract}
    \selectlanguage{english}
    We describe a new method which allows us to obtain a result of exact
    controllability to trajectories of multidimensional conservation laws in the context of entropy
    solutions and under a mere non-degeneracy assumption on the flux and a natural
    geometric condition.
\end{abstract}

\section{Introduction}
\label{sec:introduction}
In this paper we consider a scalar conservation law in several space dimensions, i.e.\ a partial differential equation of the form 
\begin{equation*}
    \label{eq:claw}
    \partial_t u +\Div_x \left(f(u)\right) = 0, \quad t\in\mathbb{R}^+, \quad x\in\Omega\subset\mathbb{R}^d, \: d\geq 1,
\end{equation*}
where $\Omega$ is an open set with smooth boundary ($\mathcal{C}^2$ is sufficient), and $f$, the flux function, is in $\mathcal{C}^1(\mathbb{R}, \mathbb{R}^d)$.

We are interested in the following controllability problem. Given an initial datum $u_0\in L^\infty(\Omega)$, a suitable target profile $u_1$ and a positive time $T$, we aim at constructing an entropy weak solution $u \in L^\infty( \mathbb{R}^+ \times \mathbb{R}^d; \mathbb{R})$ of 
\begin{equation}
    \label{eq:problem}
    \begin{cases}
        \partial_t u +\Div_x \left(f(u)\right) = 0, &\quad \text{in } (0, T)\times \Omega\\
        u(0, x) = u_0(x), &\quad \text{on }\Omega, \\
        u(T, x) = u_1(x), &\quad \text{on }\Omega, \\
    \end{cases}
\end{equation}
by using the boundary data on $(0,T)\times\partial\Omega$ as controls.

Given any extensive quantity $u$ defined on a domain $\Omega$,  such as mass or energy,  a
conservation law  for $u$ translates into a partial differential equation the physical
observation that the total amount of $u$ in $\Omega$ changes at a rate which corresponds to the
net flux of $u$, $f(u)$, through the boundary $\partial\Omega$. This kind of equations is
widely used in modeling phenomena such as gas dynamics (Euler equations), electromagnetism, magneto-hydrodynamics, shallow water, combustion, road traffic, population dynamics and  petroleum engineering. 

It is well known that even starting from initial data in $\mathcal{C}_c^\infty(\mathbb{R}^d)$ the classical solutions of \eqref{eq:claw} can develop singularities (jump discontinuities) in finite time,  see \cite{D} for a very complete introduction to the field. 

The most general wellposedness result for classical solutions to the Cauchy problem states that for any initial datum $u_0$ in $H^s$, with $s>1+\frac{d}{2}$, there exists a solution of \eqref{eq:claw} in $\mathcal{C}^0([0,T], H^s)\cap\mathcal{C}^1([0,T], H^{s-1})$, whose life span $T$ can be estimated depending on $f$ and $u_0$.

However, most of the literature devoted to conservation laws focuses on a class of weak
(distributional) solutions which  satisfies an additional selection criterium, necessary to ensure uniqueness
, called entropy condition. In the case of a scalar conservation law in several space dimensions a complete wellposedness theory for entropy solutions to the Cauchy problem is due to Kruzhkov, \cite{K}.

The problem we aim at solving, see \eqref{eq:problem}, can be adressed both in the framework of
classical or of entropy solutions. In the first case controls, besides driving the state to the
target, are also responsible of preventing the formation of singularities. Several results exist in this framework,  see \cite{Ch}, \cite{Li} and \cite{Co} for a survey. Unfortunately, this approach does not allow to attain many physically relevant states involving jump discontinuities  and leads to control strategies which are in general not  very robust. Indeed very small perturbations of the control might lead to blow up of the derivatives of the solution before time $T$.

In the present paper we are interested in controllability of entropy solutions. The literature in this framework is less abundant  also due to specific technical difficulties, even if we can notice a growing interest of researchers in this field. The classical methodology for exact controllability relies heavily on linearization which is not possible (or at least horribly technical) anymore around discontinuous solutions. Moreover, Bressan and Coclite showed in \cite{BC} that nonlinear conservation laws may fail the linear test. Indeed they provided a system for which the linearized approximation around a constant state is controllable, while the original nonlinear system cannot reach that same constant state in finite time.

A separate issue is related to the irreversibility of entropy solutions: the set of admissible target states in time $T$ is reduced and its description, often involving a number of highly technical conditions, is in itself an open problem in most cases, see \cite{ADGR, ADM, AC, AM, CM}. 


In the existing literature we can distinguish essentially three approaches toward the
study of  exact controllability for conservation laws in one space dimension (consider equation \eqref{eq:claw} with $d=1$).

Starting from the  pioneering work by Ancona and Marson, \cite{AM}, several results have been obtained  using the theory of generalized characteristics introduced by Dafermos in \cite{Daf2}, as 
\cite{ADM, AM, CM, H, P1} or the explicit Lax-Oleinik representation formula, as \cite{ADGR, AGG}. The latter technique is applicable only when the flux function $f$ is strictly convex/concave, while the theory of generalized characteristics covers also the (slightly) more general case of a flux function $f$ with one inflection point. 

The return method introduced by Coron, \cite{Co}, is the basis of the approach developed by
Horsin in \cite{H} and, combined with the classical vanishing viscosity method, plays a key
role in  \cite{GG} and in the only paper to our knowledge  in which the flux function $f$ is allowed to have a
finite number of inflection points, \cite{L}. 

The asymptotic stabilization of entropy solutions to scalar conservation laws is the topic of
\cite{BLMPB, P2, P3}. 

The only available tool for investigating the exact controllability of systems of conservation laws in one space dimension is wave front tracking algorithm, \cite{Br
},  which has been successfully applied in \cite{AC, BC, G1, G2, LY}.

The asymptotic stabilization of entropy solutions to systems has been studied in  \cite{AM2, BCo, CEGGP}.

It seems difficult to investigate the exact controllability of entropy solutions of scalar
conservation laws in several space dimensions using the techniques designed for the one
dimensional case. In the present paper we propose a new approach, inspired by the work on stabilization by Coron, \cite{Co99}, and Coron, Bastin and d'And\'ea Novel, \cite{CBAN}, see also the monography \cite{BCo} for a comprehensive presentation of the method. The conditions we impose on the flux
function are technical and will be detailed in the next Section, but we stress that in the
special case $d=1$ they are not related to  convexity (or concavity). This means that even in the one dimensional case our
result is new, althought for this case stronger results are available under more restrictive hypothesis.

The first of our conditions, called later \emph{nondegeneracy condition}, says that the range of $u$ does not contain any interval on which $f$ is affine. This condition is necessary to ensure the existence of traces at the boundary of $\Omega$, see \cite{V}.

The second condition, called lated \emph{replacement condition} involves $f$ together with $T$ and $\Omega$. Roughly speaking, once we reduce to the one-dimensional case, it says that all generalized characteristics issued from points $(t,x)$ in $\{0\}\times \Omega$ leave the cylinder $(0,T)\times \Omega$ before time $T$, so that the dynamics in the domain only depends on the boundary data and not on the initial data for $t$ large enough.    

\section{Preliminary definitions and notations}
\label{sec:preliminaries}

In the following $u\mapsto \s(u)$ is the function given by
\begin{equation*}
    \forall u\in \mathbb{R},\qquad \s(u):=
    \begin{cases}
        1 & \text{if } u>0\\
        0 & \text{if }u=0\\
        -1 & \text{if }u<0,
    \end{cases}
\end{equation*}
$\langle \cdot|\cdot \rangle$ denotes the scalar product between two vectors and $\chi_{E}$ is
the indicator function of the set $E$.

\begin{e-definition}
    \label{defi:senseIni}
    Given $f \in\mathcal{C}^1\left(\mathbb{R} ; \mathbb{R}^d\right)$ and $u_0\in L^\infty\left(\Omega\right)$, we consider the equation
    \begin{equation}\label{eq:ref}
        \partial_t u+\Div(f(u))=0, \qquad\text{for } (t,x) \in (0,+\infty)\times\Omega,
    \end{equation}
    supplemented with the initial condition
    \begin{equation}\label{eq:ic}
        u(0, x) = u_0(x),  \qquad\text{for } x\in\Omega. 
    \end{equation}

    A function $u\in L^\infty\left([0,+\infty)\times \Omega\right)$ is an entropy solution  of
    \eqref{eq:ref}--\eqref{eq:ic} in $[0,T)\times \Omega$ if for any real number $k$ and any non negative function $\phi$ in $\mathcal{C}_c^1([0,+\infty)\times\Omega)$  we have
    \begin{equation}\label{eq:entropy}
        \begin{aligned}
            \int_{(0,T)\times \Omega } |u(t,x)-k| \partial_t \phi(t,x)+\s(u(t,x)-k) & \langle
            f(u(t,x))-f(k)|\nabla \phi(t,x)\rangle dt dx \\
                                                                                    &+ \int_{\Omega} \s(u_0(x)-k) \phi(0,x) dx \geq 0.
        \end{aligned}
    \end{equation}
\end{e-definition}
We will also say that a function $u$ is an entropy solution (without refering to any initial
data)  in $(0,+\infty)\times\Omega$ when it satisfies \eqref{eq:entropy} for any non negative $\phi \in \mathcal{C}_c^1((0,+\infty)\times
\Omega)$.

We now introduce a simple geometric condition which is sufficient (though clearly not
necessary) to obtain our controllability result.
\begin{e-definition}
    Let $\Omega$ be a smooth open set of $\mathbb{R}^d$, $I$ be a segment of $\mathbb{R}$ and
    $f:\mathbb{R} \to \mathbb{R}^d$ a  $\mathcal{C}^1$ flux function.
    We say that the triple $\left(f, \Omega, I\right)$ satisfies the replacement condition in time $\mathfrak{t}>0$ if there exists a vector $w\in \mathbb{R}^d$ and a positive number $c$ such that
    \begin{align}
        L:=\sup_{x\in \Omega}\langle w|x\rangle-\inf_{x\in \Omega}\langle w|x\rangle<+\infty,\\
        \forall u\in I,\qquad \langle f'(u)|w\rangle \geq c,  \quad\text{and } \mathfrak{t}=\frac{L}{c}.
    \end{align}
\end{e-definition}
\begin{e-definition}
    We say that the flux $f$ is non degenerate if for any couple 
    $(\tau, \zeta)\in \mathbb{R}\times \mathbb{R}^d \setminus \{(0,0)\}$ we have
    \begin{equation*}
        \mathcal{L}\left(\{z\in \mathbb{R}\ :\ \tau+\langle \zeta|f'(z)\rangle=0\}\right)=0,
    \end{equation*}
    where $\mathcal{L}$ is the Lebesgue measure.
\end{e-definition}
We can now state our main theorem  on exact controllability to trajectories for a conservation law in any space dimension.
\begin{thm}\label{th:main}
    Let $v\in \mathcal{C}^0((0,+\infty); L^1(\Omega))\cap L^\infty((0,+\infty)\times \Omega)$ be an
    entropy solution to \eqref{eq:ref} and $u_0$ be a function in $L^\infty(\Omega)$.

    We suppose that both $u_0$ and $v$ take values in a segment $I$ such that $\left(f, \Omega, I\right)$ satisfies the replacement condition in time $\mathfrak{t}$. We also suppose that the flux $f$ is non degenerate.

    Then for any times $T_1$ and $T_2$ larger than $\mathfrak{t}$, there exists an entropy solution 
    $u$ of \eqref{eq:ref} satisfying
    \begin{equation*}
        u(0,x)=u_0(x),\qquad u(T_1,x)=v(T_2,x)\qquad \text{for almost every }x\in \Omega.
    \end{equation*}
\end{thm}
\begin{remark}
    For the sake of simplicity, we omit to write here the exact form of the controls we use. In
    the next Section we precise in which sense the boundary condition on $\partial\Omega$ are
    taken into account by entropy solutions and in the last Section, in the proof of Theorem~\ref{th:main}, we write our controls in a fully explicit way.

\end{remark}
\begin{remark}
    A characterization of the set of admissible target profiles at fixed time  time $T\geq 0$ for
    a scalar conservation law in several space dimensions is not available in the literature. We
    stress however that in the statement of Theorem~\ref{th:main} we really need to assume that $v$ is an entropy solution on the cylinder $(0,+\infty)\times \Omega$ because the complete knowledge of $v$ is necessary in our proof.   
\end{remark}

\section{Boundary conditions and entropy solutions}
\label{sec:boundary}

We have so far avoided the precise formulation of boundary conditions for \eqref{eq:ref}. In general, given the initial boundary value problem 
\begin{equation}\label{eq:sysBord}
    \begin{cases}
        \partial_t u + \Div(f(u))=0,\qquad& (0,+\infty)\times \Omega,\\
        u(t,x)=u_b(t,x),\qquad   &(t,x)\in (0,+\infty)\times \partial \Omega,\\
        u(0,x)=u_0(x),\qquad    & x\in \Omega,
    \end{cases}
\end{equation}
its entropy solution $u$  does not satisfy the boundary condition in the usual sense, as the trace of $u$ on $\partial \Omega$ does not coincide with the prescribed Dirichlet datum. The situation is easier to visualize in the one dimensional case as we can see in the following exemple.
\begin{exemple}
    Assume $d=1$, $\Omega = (0,+\infty)$, $f(u) = \dfrac{u^2}{2}$ and impose in \eqref{eq:sysBord}  constant initial and boundary data,  $u_0 = -1$ and $u_b=-\dfrac{1}{2}$. Then the initial condition is transported along characteristic curves with negative slope up to the boundary, while no characteristic can spring out of the boundary itself. The trace of the solution at $x= 0$ can only take the value $u(t, 0^+) = u_0$, and  $u_b$ can not be attained.    
\end{exemple}
For this reason the boundary conditions should be interpreted in a broader sense, made precise
by Leroux, \cite{Le}, and Bardos, Leroux and N\'ed\'elec, \cite{BLN}. In the setting of the
exemple above we say that the boundary condition is fulfilled  in the sense of
Bardos-Leroux-N\'ed\'elec as soon as the solution  to the Riemann problem with data $u_L = u_b$
and $u_R = u(t, 0^+)$ only contains waves of non positive speed (i.e.\ waves which do not enter the domain).  
In the general multidimensional case this takes the following form.
\begin{e-definition}
    Let $I(a,b)$ denote the interval of extrema $a$ and $b$, and let $\eta(x)$ be the outer unit normal of $\partial \Omega$ at $(t,x)\in (0,T)\times \partial \Omega$. Then we say that the boundary condition $u_b$ in the IBVP \eqref{eq:sysBord} is fulfilled at $(t,x)\in (0,T)\times \partial \Omega$ if for any $k \in I(u_b(t,x), u(t,x))$
    \begin{equation*}\label{eq:BNL}
        \mathrm{sign}(u(t,x) - u_b(t,x)) \left( f(u(t,x))\cdot\eta(x) - f(k)\cdot \eta(x)\right) \geq 0. 
    \end{equation*}
\end{e-definition}
We have to precise, however, that the above definition is not exactly the one we adopt in the present work as existence of traces is not guaranted for solution of \eqref{eq:sysBord} in the $L^\infty$ setting. The first results dealing
with this problem were by Otto, \cite{Ot}, see also \cite{MNR}. We use more recent results by Ammar, Carillo and Wittbold, \cite{ACW}, which build upon those ideas. We also recall a (simplified version of a) result by Vasseur, \cite{V}, showing that if the flux satisfies the non degeneracy condition then any entropy solution $u\in L^\infty$ admits a trace at the boundary. 

Let us recall some definitions and results in \cite{ACW}. We use the following notations.
For any real numbers $\alpha$ and $k$, and any point $x\in \partial \Omega$, we call  $\eta(x)$ the outer unit normal at $x$ and introduce 
\begin{equation*}
    \omega^+(x,k,\alpha):=\max_{k\leq r,s\leq \max(\alpha,k)}|\langle f(r)-f(s)|\eta(x)\rangle|,
\end{equation*}
\begin{equation*}
    \omega^-(x,k,\alpha):=\max_{\min(\alpha,k)\leq r,s\leq k}|\langle f(r)-f(s)|\eta(x)\rangle|.
\end{equation*}
For integrals on the boundary we denote the surface measure  at $x\in \partial \Omega$ by $d\sigma(x)$.
\begin{e-definition}\label{defi:senseComplet}
    Given a boundary condition $u_b\in L^\infty((0,+\infty)\times \partial \Omega)$ and an initial
    data $u_0\in L^\infty(\Omega)$ we say that $u$ is an entropy solution to \eqref{eq:sysBord} 
    when the following hold for any $k\in\mathbb{R}$ and any non negative function $\zeta\in \mathcal{C}^\infty_c([0,+\infty)\times \mathbb{R}^d)$
    \begin{multline}\label{eq:plus}
        \int_0^T\int_\Omega  {(u(t,x)-k)^+\partial_t \zeta(t,x)+\chi_{\{u(t,x)>k\}}\langle f(u(t,x))-f(k))|\nabla \zeta(t,x)\rangle \,dx \,dt}\\
        +\int_{\partial \Omega} \int_0^T {\zeta(t,x)\omega^+(x,k,u_b(t,x)) \, dt\, d\sigma(x)}+\int_\Omega {(u_0(x)-k)^+ \zeta(0,x)\, dx}\geq 0,
    \end{multline}

    \begin{multline}\label{eq:moins}
        \int_0^T \int_\Omega {(k-u(t,x))^+\partial_t \zeta(t,x)+\chi_{\{u(t,x)<k\}}\langle f(u(t,x))-f(k))|\nabla \zeta(t,x)\rangle \,dx \, dt }\\
        +\int_{\partial \Omega}\int_0^T {\zeta(t,x)\omega^-(x,k,u_b(t,x)) \,dt\, d\sigma(x)}+\int_\Omega {(k-u_0(x))^+ \zeta(0,x)\, dx}\geq 0.
    \end{multline}
\end{e-definition}
The following two theorems were proven in \cite{ACW}  (see \cite{ACW}, Theorem 2.3 and 2.4).
\begin{thm}\label{thm:ACWexistence}
    Given initial and boundary data  $u_0\in L^\infty(\Omega)$ and  $u_b\in L^\infty((0,+\infty)\times \partial \Omega)$,
    there exists a unique entropy solution of \eqref{eq:sysBord}.
\end{thm}
\begin{thm}
    Given initial data $u_0,\, v_0$ in $L^\infty(\Omega)$ and boundary  data $u_b,\, v_b$ in  $L^\infty((0,+\infty)\times \partial \Omega)$, the corresponding entropy solutions $u$ and $v$ satisfy
    \begin{multline}\label{eq:comparaison}
        \int_0^T \int_\Omega {(u(t,x)-v(t,x))^+\partial_t \zeta(t,x) + \chi_{\{u(t,x)> v(t,x)\}}\langle f(u(t,x))-f(v(t,x)))|\nabla \zeta(t,x)\rangle \,dx\,dt}\\
        +\int_{\partial \Omega}\int_0^T{\omega^-(x,u_b(t,x),v_b(t,x))\zeta(t,x)\,dt\,d\sigma(x)} +\int_\Omega
        {(u_0(x)-v_0(x))^+\zeta(0,x)\,dx}\geq 0.
    \end{multline}
    for any non negative function $\zeta\in \mathcal{C}^\infty_c([0,+\infty)\times \mathbb{R}^d)$.
\end{thm}

Let us finally recall a simplified version of the result obtained by Vasseur in \cite{V},  which is sufficient for our use.
\begin{thm}\label{thm:Vasseur}
    Assume  that the flux $f$ is non degenerate and that the domain $\Omega$ is $\mathcal{C}^2$. 
    Then if $u\in L^\infty((0,+\infty)\times \Omega)$ is an entropy solution of \eqref{eq:ref} in the sense of
    Definition~\ref{defi:senseIni}, i.e. \eqref{eq:entropy} is satisfied for any $k$ and any non negative function $\phi$ in $\mathcal{C}_c^1((0,+\infty)\times\Omega)$, then there exists boundary data $u_b\in L^\infty((0,T)\times \partial \Omega)$ and initial data $u_0\in L^\infty(\Omega)$ such that $u$ is the unique entropy solution of the mixed problem \eqref{eq:sysBord} in the sense of Definition~\ref{defi:senseComplet}.
\end{thm}

\section{Proof of the main result}
\label{sec:proof}

\begin{lemma}\label{lem:stabiliteIntervalle}
    Consider $J:=[A,B] \subset \mathbb{R}$ and suppose that
    \begin{equation*}
        \begin{aligned}
      & u_0(x) \in J, \quad \text{for a.e. } x\in\Omega, \\
      & u_b(t) \in J, \quad \text{for a.e. } t\geq 0.
        \end{aligned}
    \end{equation*}
    Then the unique entropy solution to the IBVP \eqref{eq:sysBord} with initial
    and boundary  data $u_0$ and  $u_b$, $u$  satisfies
    \begin{equation*}
        u(t, x) \in J \quad \text{for a.e. } (t,x)\quad\text{in}\quad (0,+\infty)\times \Omega.
    \end{equation*}
\end{lemma}

\begin{proof} We prove here in full details that $u(t,x)\leq B$  for a.e. $(t,x)$ in  $(0,+\infty)\times \Omega$. The inequality $A\leq u(t,x)$ can be obtained analogously.

    By hypothesis we have for almost every $x$ in $\Omega$ and $(t,y)$ in $(0,+\infty)\times
    \partial \Omega$ 
    \begin{equation*}
        u_0(x)\leq B,\qquad u_b(t,y)\leq B,
    \end{equation*}
    then for  any fixed  time $\tilde t\geq 0 $, taking a sequence  $\zeta_n\in
    \mathcal{C}_c^\infty(\mathbb{R})\to \chi_{(-\infty,\tilde t]}$
    and using $k=B$,  from \eqref{eq:plus} we obtain
    \begin{equation}\label{eq:prooflemma}
        \int_{\partial \Omega}\int_0^{\tilde t} {\omega^+(y,B,u_b(t,y))\,dt\,d\sigma(y)}
        +\int_\Omega{(u_0(x)-B)^+-(u(\tilde t,x)-B)^+ \,dx}\geq 0. 
    \end{equation}
    It is clear that for a.e.  $x$ in $\Omega$ and $(t,y)$ in $(0,\tilde{t})\times \partial
    \Omega$ we have
    \begin{equation*}
        \omega^+(y,B,u_b(t,y))=0,\qquad (u_0(x)-B)^+=0,
    \end{equation*}
    then \eqref{eq:prooflemma} implies 
    \begin{equation*}
        (u(\tilde t,x)-B)^+=0,\quad \text{for a.e. $x$ in $\Omega$}
    \end{equation*}
    which is indeed
    \begin{equation*}
        u(\tilde t,x)\leq B,\quad\text{for a.e. $x$ in $\Omega$}.
    \end{equation*}
\end{proof}

\begin{e-proposition}\label{prop:syncro}
    Let $u$ and $v$ be entropy solutions of  \eqref{eq:sysBord}  with respective initial data $u_0$ and $v_0$ and the same boundary
    datum $u_b$.    Let us also suppose that all data take value in an interval $I$ which
    satisfies the
    replacement condition in time $\mathfrak{t}=\frac{L}{c}$ (with a direction $w$). Then we can conclude that
    \begin{equation*}
        \forall t\geq \mathfrak{t},\qquad u(t,x)=v(t,x),\qquad  \text{for almost every $x$ in $\Omega$}.
    \end{equation*}
\end{e-proposition}

\begin{proof}
    Let us define for $\theta>0$ the functional $J_\theta$ by
    \begin{equation*}
        \forall t\geq 0,\qquad J_\theta(t):=\int_\Omega{|u(t,x)-v(t,x)|e^{-\theta \langle
        w|x\rangle }dx}.
    \end{equation*}
    Given $\bar t \geq 0$, we apply \eqref{eq:comparaison} to the ordered couples $(u,v)$ and  $(v,u)$, then  adding the inequalities we
    get
    \begin{multline*}
       \int_0^{\bar t} \int_\Omega
      \left((v(t,x)-u(t,x))^++(u(t,x)-v(t,x))^+\right)\partial_t\zeta(t,x)\\
       + \langle (\chi_{\{v(t,x)> u(t,x)\}} f(v(t,x))-f(u(t,x)))+(\chi_{\{u(t,x)> v(t,x)\}} f(u(t,x))-f(v(t,x)))|\nabla
      \zeta(t,x)\rangle \,dx\,dt\\
      +\!\int_{\partial \Omega}\!\int_0^{\bar t}
      {2\omega^-(x,u_b(t,x),u_b(t,x))\zeta(t,x)\,dt\,d\sigma(x)}\\
      +\!\int_\Omega {((v_0(x)-u_0(x))^++(u_0(x)-v_0(x))^+)\zeta(0,x)dx}\!\geq 0.
    \end{multline*}
    which is actually
    \begin{multline*}
        \int_0^{\bar t} \int_\Omega{|u(t,x)\!-\!v(t,x)|\partial_t \zeta(t,x)+ \s(u(t,x)\!-\!v(t,x))\langle
        f(u(t,x))\!-\!f(v(t,x))|\nabla \zeta(t,x)\rangle \,dx\,dt}\\
        + \int_\Omega{|u_0(x)-v_0(x)|\zeta(0,x)\,dx}\geq 0.
    \end{multline*}
    We consider a sequence  $\left(\zeta_n\right)_n\subset \mathcal{C}_c^\infty(\mathbb{R})$
    converging in $L^1$ to  $\chi_{(-\infty,\bar t]}e^{-\theta\langle w|x\rangle}$, so that in the limit $n\to \infty$ we get 
    \begin{equation}\label{eq:egalite}
        J_\theta(\bar t)\leq J_\theta(0)+\int_0^{\bar t}{\int_\Omega{\s(u(t,x)-v(t,x))\langle f(u(t,x))-f(v(t,x)))|-w
        \theta e^{-\theta\langle w|x\rangle}\rangle \,dx}\,dt}.
    \end{equation}
    But since 
    \begin{align*}
        \forall (a,b)\in I^2,\qquad  \s(a-b)\langle f(a)-f(b)|w\rangle
        &=\s(a-b)\langle \int_0^1{f'(b+s(a-b))ds}(a-b)|w\rangle\notag\\
        & =\s(a-b)(a-b)\int_0^1{\langle f'(b+s(a-b))|w\rangle ds}\notag\\
        &\geq|a-b|\int_0^1{c\ ds}\notag\\
        &\geq c |a-b|\label{eq:calcul},
    \end{align*}
    from \eqref{eq:egalite}, Lemma~\ref{lem:stabiliteIntervalle} and the replacement
    condition, we obtain 
    \begin{equation*}
        J_\theta(\bar t)\leq J_\theta(0)-\theta c\int_0^{\bar t}{J_\theta(s)ds}.
    \end{equation*}
    Thanks to the classical Gronwall's lemma we end up with
    \begin{equation*}
        J_\theta(\bar t)\leq e^{-c\theta \bar t} J_\theta(0).
    \end{equation*}
    As $\bar t$ was arbitrarily chosen, if  $M:=\underset{x\in \Omega}{\sup}\ \langle w|x\rangle $ and $m=\underset{x\in \Omega}{\inf}\ \langle w|x\rangle $, we can write that for all $t\geq 0$
    \begin{equation*}
        \|u(t)-v(t)\|_{L^1(\Omega)}e^{-\theta M}\leq J_\theta(t)\leq
        \|u(t)-v(t)\|_{L^1(\Omega)}e^{-\theta m}.
    \end{equation*}
    So we can compute 
    \begin{align*}
        \|u(t)-v(t)\|_{L^1(\Omega)}&\leq e^{\theta M}J_\theta(t)\\
                                   &\leq e^{\theta M-\theta c t} J_\theta(0)\\
                                   & \leq e^{-\theta c(\frac{M-m}{c}-t)}\|u_0-v_0\|_{L^1(\Omega)}\\
                                   & \leq e^{-\theta c(\frac{L}{c}-t)}\|u_0-v_0\|_{L^1(\Omega)}.
    \end{align*}
    So for any $t\geq \mathfrak{t} = \dfrac{L}{c}$ letting $\theta \to +\infty$ we obtain
    \begin{equation*}
        u(t,x)=v(t,x)\qquad \text{for almost every $x$ in $\Omega$}.
    \end{equation*}
\end{proof}

We are ready to prove Theorem~\ref{th:main}.

\begin{proof}
    We aim at proving that there exists an entropy solution to the problem
    \begin{equation*}
        \begin{cases}
            \partial_t u +\Div_x \left(f(u)\right) = 0, &\quad \text{in } (0, T_1)\times \Omega\\
            u(0, x) = u_0(x), &\quad \text{on }\Omega, \\
            u(T_1, x) = v(T_2, x), &\quad \text{on }\Omega. \\
        \end{cases}
    \end{equation*}

    In view of the well-posedness result stated in Theorem~\ref{thm:ACWexistence} our goal is achieved once we construct suitable boundary conditions, which can be interpreted as controls in our setting.

    \noindent\textbf{Case $T_2>T_1$. }

    Thanks to Theorem~\ref{thm:Vasseur} it makes sense to consider 
    \begin{align*}
     & w_0(x) = v(T_2-T_1, x), &\qquad \text{for a.e. } x\in\Omega, \\
     & w_b(s, x ) = v(T_2-T_1 + s, x),  &\qquad \text{for a.e. } x\in \partial\Omega\quad \text{and } s\geq 0.
    \end{align*}

    We call $w$ the unique entropy solution to the IBVP \eqref{eq:sysBord} with data $w_0$, $w_b$ on $(0,T_1)\times \Omega$. The form of the equation implies that for almost every $(s,x)$ in $(0,T_1)\times \Omega$, $w(s,x)  = v(T_2-T_1+s, x)$.

    By hypothesis $T_1\geq \mathfrak{t}$ so, as a direct application of Proposition~\ref{prop:syncro}, we can conclude that the entropy solutions to the mixed problems of the form \eqref{eq:sysBord}  with initial data $u_0$ and $w_0$ respectively,  and common boundary datum $w_b$ satisfy
    \begin{equation*}
        u(T_1, x) = w(T_1, x) \qquad \text{for a.e.} x\in\Omega,
    \end{equation*}
    which means
    \begin{equation*}
        u(T_1, x) = v(T_2, x) \qquad \text{for a.e.} x\in\Omega.
    \end{equation*}

    \noindent\textbf{Case $T_1>T_2$. } 

    We define
    \begin{align*}
    & w_0(x) = v(T_2-\mathfrak{t}, x), &\qquad \text{for a.e. } x\in\Omega, \\
    & w_b(s, x ) = v(T_2-\mathfrak{t} + s, x),  &\qquad \text{for a.e. } x\in
    \partial\Omega\quad \text{and } s\geq 0,
    \end{align*}
    where $\mathfrak{t}$ is the time given by the replacement condition.

    We call $w$ the unique entropy solution to the IBVP \eqref{eq:sysBord} with data $w_0$, $w_b$ on $(0,\mathfrak{t})\times \Omega$. The form of the equation implies that for almost every $(s,x)$ in $(0,\mathfrak{t})\times \Omega$, $w(s,x)  = v(T_2-\mathfrak{t}+s, x)$.

    We consider now a boundary condition  of the following form
    \begin{equation*}
        u_b(t, x) = 
        \begin{cases}
            \mathfrak{b}, \qquad &\text{for } t \in (0, T_1- \mathfrak{t}) , \quad x\in\partial\Omega,\\
            w_b(t-(T_1- \mathfrak{t}), x ) \qquad &\text{for } t \in ( T_1- \mathfrak{t}, +\infty) , \quad x\in\partial\Omega,  
        \end{cases}
    \end{equation*}
    where $\mathfrak{b}$ is any constant state in the interval $I$.

    The IBVP \eqref{eq:sysBord} with data $u_0$, $u_b$ admits a unique entropy solution $u$ in $(0,+\infty) \times\Omega$.

    We call $\tilde u_0$ the profile of $u$ at time $ t = T_1- \mathfrak{t}$.

    Now it is clear that if we apply Proposition~\ref{prop:syncro} to the entropy solutions $\tilde u$ et $w$ of \eqref{eq:sysBord} with respective initial data $\tilde u_0$ et $w_0$ and common boundary data $w_b$ we obtain
    \begin{equation*}
        \tilde u (\mathfrak{t}, x) = w(\mathfrak{t}, x), \qquad \text{for a.e. } x\in\Omega,
    \end{equation*}
    which means
    \begin{equation*}
        u (T_1, x) = v(T_2, x), \qquad \text{for a.e. } x\in\Omega.
    \end{equation*}

\end{proof}

\section*{Acknowledgements}
The first author acknowledges the financial support of the Région Franche-Comté, projet \emph{Analyse mathématique et simulation numérique d'EDP issus de problèmes de contrôle et du trafic routier}. The second author was supported by ANR Finite4SOS (15-CE23-0007).

\end{document}